\documentclass[10pt,a4paper]{amsart}
\usepackage[T1]{fontenc}
\usepackage[utf8]{inputenc}
\usepackage{amsmath,amssymb,mathtools,mathrsfs}
\usepackage[textwidth=16cm,textheight=22cm,centering]{geometry}
\usepackage[shortlabels]{enumitem}
\usepackage{aliascnt}
\usepackage{etoolbox}
\usepackage[nocompress]{cite}
\usepackage[colorlinks=true,citecolor=red,linkcolor=blue,urlcolor=blue]{hyperref}

\theoremstyle{plain}
\newtheorem{thm}{Theorem}[section]
\newaliascnt{prop}{thm}
\newtheorem{prop}[prop]{Proposition}
\aliascntresetthe{prop}
\newaliascnt{lem}{thm}
\newtheorem{lem}[lem]{Lemma}
\aliascntresetthe{lem}
\theoremstyle{definition}
\newaliascnt{defn}{thm}

\aliascntresetthe{defn}
\theoremstyle{remark}
\newaliascnt{rem}{thm}
\newtheorem{rem}[rem]{Remark}
\aliascntresetthe{rem}

\newcounter{proofnumber}
\newcounter{stp}
\theoremstyle{definition}
\newtheorem{step}[stp]{Step}

\AtBeginEnvironment{proof}{\stepcounter{proofnumber}\setcounter{stp}{0}}
\numberwithin{equation}{section}
\setlist[enumerate]{font=\normalfont,label=(\roman*),leftmargin=*}
\allowdisplaybreaks[2]

\newcommand{\R}{\mathbb{R}}

\renewcommand{\H}{\mathrm{H}}
\newcommand{\rL}{\mathrm{L}}
\newcommand{\rX}{\mathrm{X}}
\newcommand{\rE}{\mathrm{E}}
\newcommand{\rH}{\H}
\newcommand{\rC}{\mathrm{C}}
\newcommand{\rCinfty}{\rC^\infty}

\newcommand{\sA}{\mathcal{A}}

\newcommand{\sD}{\mathcal{D}}

\newcommand{\sL}{\mathcal{L}}

\newcommand{\nablaH}{\nabla_{\H}}
\DeclareMathOperator{\divH}{div_{\H}}

\renewcommand{\d}{\,\mathrm{d}}
\newcommand{\dt}{\partial_t}
\newcommand{\dz}{\partial_z}

\newcommand{\T}{\mathbb{T}}
\newcommand{\rW}{\mathrm{W}}

\newcommand{\sF}{\mathcal{F}}

\newcommand{\EE}{\mathbb{E}}
\newcommand{\PP}{\mathbb{P}}

\title{Global well-posedness of the primitive equations with stochastic wind-driven boundary conditions}
\date{}
\hypersetup{pdftitle={Global well-posedness of the primitive equations with stochastic wind-driven boundary conditions}}
\keywords{Primitive equations, stochastic wind driven boundary conditions,
global pathwise well-posedness, stochastic convolution, maximal regularity,
anisotropic Sobolev spaces.\\
Tarek Z\"ochling gratefully acknowledges the support of the
Deutsche Forschungsgemeinschaft (DFG) through the Research Unit FOR~5528.
He thanks  Antonio Agresti, Tim Binz, Matthias Hieber and Amru Hussein for fruitful discussions
on this topic.}
\author{Tarek Z\"{o}chling\textsuperscript{1}}
\address{\textsuperscript{1}Technische Universit\"{a}t Darmstadt,
Schlo\ss{}gartenstra{\ss}e 7, 64289 Darmstadt, Germany.}
\email{zoechling@mathematik.tu-darmstadt.de}

\begin{document}

\begin{abstract}
Consider the three-dimensional primitive equations subject to Neumann wind stress driven by finitely many
Brownian motions. For arbitrary smooth spatial wind profiles, global pathwise
existence and uniqueness without assuming a vertical derivative of the initial datum is proved,
extending the local theory of Binz, Hieber, Hussein, and Saal~\cite{BHHS-24}
for this class of data and wind profiles. The proof combines a local maximal
$\rL^2$ regularity construction with weighted estimates for the boundary
stochastic convolution.
\end{abstract}
\maketitle

\section{Introduction}

The primitive equations are a standard model for large-scale oceanic and
atmospheric flows, obtained from the Navier--Stokes equations by replacing
the vertical momentum equation with the hydrostatic balance.
Their mathematical theory was developed by Lions, Temam, and Wang
in a series of papers on the atmosphere, the ocean, and their
coupling~\cite{LTW-atmosphere-92,LTW-92,LTW-CAO-93,LTW-CAO-95}.

In this article, we consider the three-dimensional incompressible
primitive equations with full viscosity on a horizontally periodic
layer, subject to a stochastic Neumann wind stress at the upper
boundary and a homogeneous Neumann condition at the bottom.
A local pathwise theory for primitive equations with stochastic wind
driven boundary conditions was established in~\cite{BHHS-24}.
Global continuation was left open there because suitable a priori
estimates in the anisotropic spaces used for the local construction
were unavailable. We address this question for wind stresses generated
by finitely many Brownian motions with smooth spatial profiles.

Our main result, stated in \autoref{thm-global-wind}, establishes global
existence and pathwise uniqueness for suitable large deterministic
initial velocities. The wind profiles are deterministic and independent
of time, and their amplitudes need not be small. This gives a global
pathwise theory for the stated class of initial data and boundary noises,
addressing the global existence question raised in~\cite{BHHS-24}.
For the related two-dimensional
Navier--Stokes problem, global pathwise weak well-posedness under
stochastic Neumann boundary forcing was proved in~\cite{AL-24}.
That result allows spatially rough wind profiles and arbitrary
square-integrable solenoidal initial data, whereas the additional
horizontal regularity used here supports the analysis of the
three-dimensional hydrostatic system.

For the deterministic primitive equations, global weak existence was
established in~\cite{LTW-92}, and Cao and Titi~\cite{CT-07} proved
global existence and uniqueness of strong solutions for arbitrarily
large $\rH^1$ initial data in the classical cylindrical setting.
An $\rL^q$ approach based on the hydrostatic Helmholtz projection
and the hydrostatic Stokes operator was subsequently developed
in~\cite{HK-16}, allowing initial data in suitable interpolation spaces
with less differentiability than $\rH^1$. Anisotropic theories also
admit initial velocities without vertical differentiability,
see~\cite{GGHHK-21}. For general background on geophysical flows and
rotating fluids we refer to~\cite{CDGG-06}, and for the
Navier--Stokes equations to~\cite{Galdi-11,LR-16}.

Wind forcing is also part of the coupled atmosphere--ocean problem,
in which both fluid velocities are unknown and the tangential stress
at the interface depends on their difference. Up to positive physical
and density factors collected in $\kappa$, the physical wind law
in~\cite{LTW-CAO-93,LTW-CAO-95} takes the form
\begin{equation*}
\tau=\kappa|v^{\mathrm a}-v^{\mathrm o}|(v^{\mathrm a}-v^{\mathrm o}).
\end{equation*}
Here the atmospheric and oceanic velocities are evaluated at the
interface. Global strong well-posedness for large data, including
critical Besov data, was proved in~\cite{BBHZ-25} for the CAO system
with this nonlinear coupling. That formulation uses atmospheric pressure
coordinates and the incompressible structure obtained by assuming
constant atmospheric pressure at the interface. Retaining atmospheric
compressibility instead couples compressible primitive equations for
the atmosphere to incompressible primitive equations for the ocean.
This compressible--incompressible CAO problem with the physical wind
law is considered in~\cite{Z-CAO}.

The compressible primitive equations have a related well-posedness
theory, including local strong solutions~\cite{LT-21} and the zero Mach
number limit for well-prepared data~\cite{LT-20}. The hydrostatic
Lagrangian approach in~\cite{HIRZ-25} gives local strong well-posedness
for large data with strictly positive density and global strong
well-posedness near equilibrium. For a model with heat conduction and
gravity, strong well-posedness near equilibrium on prescribed finite
time intervals is established in~\cite{Z-heat-26}.

The interface law also reflects the effect of atmospheric and oceanic
boundary layers. Related analyses of rotating boundary layers and
resonant wind forcing can be found in~\cite{DGV-17,DSR-09}.
Prescribing the atmospheric velocity and replacing the relative
velocity in the stress by that prescribed velocity removes the feedback
from the ocean, so this simplification requires a separate modelling
justification. The effect of friction in wind-driven shallow flows and
an asymptotic derivation of a Navier boundary condition for the primitive
equations are studied in~\cite{BS-01,BGMR-03}. In the present model,
we prescribe the stress directly as a noise that is white in time
and smooth in the horizontal variables.

The proof starts with the Da Prato--Debussche
decomposition~\cite{DPD-03}, also used for the boundary problems
in~\cite{BHHS-24,AL-24}. We write $V=v+Z$, where the stochastic
convolution $Z$ solves the linear hydrostatic Stokes problem with the
prescribed boundary noise and zero initial value. The remainder $v$
then satisfies a random evolution equation with homogeneous boundary
conditions. Additional horizontal regularity in the ground space and
product estimates with opposite vertical orders yield the local maximal
$\rL^2$ regularity theory in \autoref{prop-wind-local-hilbert}.

The key estimate for global continuation concerns the weighted vertical
derivative of $Z$. Even for smooth spatial profiles, the boundary noise
limits the vertical Sobolev regularity of $Z$ to orders below one half.
On the vertical interval $(-h,0)$, multiplication by the distance $-z$
to the noisy wall compensates for the singularity of $\dz Z$.
More precisely, for $1/4<s<1/2$ and every finite $T>0$,
\autoref{prop-wind-path-regularity} yields, almost surely, that
\begin{equation*}
Z\in\rC([0,T],\rH_z^s\rH_\H^3\cap\rL_z^6\rH_\H^3)
\ \text{ and } \ (-z)\dz Z\in\rC([0,T],\rL_z^6\rH_\H^3).
\end{equation*}
These continuous path estimates follow from bounds on the boundary
heat kernel, its weighted vertical derivative, and their time
increments. They supply control of a full vertical derivative after
multiplication by the distance to the boundary.

The vertical velocity $w(v)$ associated with the remainder vanishes
at the upper wall. The Hardy inequality \eqref{eq-wind-hardy} therefore
gives, for $1<\ell<\infty$, that
\begin{equation*}
\|\frac{w(v)}{-z}\|_{\rL^\ell}
\le\frac\ell{\ell-1}\|\divH v\|_{\rL^\ell}
\ \text{ and } \ w(v)\dz Z=\frac{w(v)}{-z}\bigl((-z)\dz Z\bigr).
\end{equation*}
The spatial norms here are taken on the periodic layer. This
factorization combines the vanishing vertical velocity with the
weighted regularity of $Z$ to control the singular term $w(v)\dz Z$.
Applied to the baroclinic component and combined with horizontal
integration by parts, it gives the wind contribution to the
sixth-power estimate in \autoref{step-wind-baroclinic-six} of the proof
of \autoref{prop-wind-global-energy}.

We can then establish the successive global energy estimates for $v$
following the order in~\cite{CT-07}, with the terms containing $Z$
treated explicitly. Starting at a positive regular time, these bounds
continue the remainder in the classical strong class on every finite
time interval. The comparison argument identifies this continuation
with the initial anisotropic solution on their overlap and shows that
all choices of restart time give the same global solution.
The remainder is consequently strong at every positive time, while
global persistence of the original anisotropic maximal regularity
norm is not asserted.

This article is structured as follows. In \autoref{sec-wind-setting},
we formulate the stochastic primitive equations and state the main
theorem with its precise solution class. In
\autoref{sec-wind-stochastic}, we construct the boundary forcing
through an explicit Neumann lifting and establish the weighted
regularity of the stochastic convolution.
\autoref{sec-wind-local} develops the local maximal regularity
theory, the Hardy estimate, and the comparison with classical strong
solutions. In \autoref{sec-wind-energy}, we derive the global
$\rL^\infty_t\rH^1\cap\rL^2_t\rH^2$ bounds for the remainder
started from a positive regular time. Finally,
\autoref{sec-wind-proof} combines these results to construct the
global continuation and prove its independence of the restart time,
adaptedness, and pathwise uniqueness.

\section{Setting and main result}
\label{sec-wind-setting}

We consider the fully viscous primitive equations on a periodic layer,
with stochastic wind stress at the upper boundary and homogeneous
Neumann conditions at the bottom.
For $h,T>0$, we define the spatial domain and the horizontal differential operators by
\begin{equation*}
\Omega=\T^2\times(-h,0),\quad \T^2=(\R/\mathbb Z)^2,\quad
\nablaH=(\partial_{x_1},\partial_{x_2})^\top,\quad
\divH u=\partial_{x_1}u_1+\partial_{x_2}u_2
\ \text{ and } \ \Delta_\H=\partial_{x_1}^2+\partial_{x_2}^2.
\end{equation*}
We write $x=(x_\H,z)$ and fix a filtered probability space
$(\Xi,\sF,(\sF_t)_{t\ge0},\PP)$ satisfying the usual conditions,
independent real Brownian motions $\beta_1,\ldots,\beta_N$, and
 deterministic profiles $g_j\in\rCinfty(\T^2)^2$ without a mean or smallness condition.
With viscosity normalized to one, the horizontal velocity $V$, vertical velocity $W$,
and pressure $P_s$ satisfy
\begin{equation}\label{eq-stochastic-pe}
\left\{
\begin{aligned}
\dt V-\Delta_\H V-\dz^2V+(V\cdot\nablaH)V+W\dz V+\nablaH P_s&=0,
&&\text{in }(0,T)\times\Omega,\\
\dz P_s&=0,
&&\text{in }(0,T)\times\Omega,\\
\divH V+\dz W&=0,
&&\text{in }(0,T)\times\Omega,\\
V(0)&=V_0,
&&\text{in }\Omega.
\end{aligned}
\right.
\end{equation}
For $0<t<T$, the impermeability and wind-stress conditions are given by
\begin{equation}\label{eq-stochastic-boundary}
W|_{z=-h,0}=0,\quad \dz V|_{z=-h}=0
\ \text{ and } \ \dz V|_{z=0}=\sum_{j=1}^Ng_j\dot\beta_j.
\end{equation}
The divergence condition and impermeability determine $W$ from $V$ through
\begin{equation*}
W(t,x_\H,z)=-\int_{-h}^z\divH V(t,x_\H,\zeta)\d\zeta
=\int_z^0\divH V(t,x_\H,\zeta)\d\zeta.
\end{equation*}
All fields are horizontally periodic, with $V$ taking values in $\R^2$ and $W,P_s$
being scalar. The white-noise boundary condition is understood through the stochastic
convolution and the boundary identity \eqref{eq-lift-weak} below. 

To rewrite the problem in operator form and apply the Da Prato--Debussche decomposition, we introduce the hydrostatic Helmholtz projection and the associated hydrostatic Stokes operator.
Let $P_\H$ denote the periodic Helmholtz projection on $\rL^2(\T^2)^2$,
with constant fields included in its range. We define the hydrostatic Helmholtz
projection $P$ on $\rL^2(\Omega)^2$ by
\begin{equation*}
Pu=u-\overline u+P_\H\overline u,
\end{equation*}
where $\overline f$ denotes the vertical average, given by
\begin{equation*}
\overline f(x_\H)=\frac1h\int_{-h}^0f(x_\H,z)\,\d z.
\end{equation*}
We then define the space of hydrostatically solenoidal fields by
\begin{equation*}
\rL_{\bar\sigma}^2(\Omega)
=P\rL^2(\Omega)^2
=\{u\in\rL^2(\Omega)^2\mid
\divH\overline u=0\text{ in }\sD'(\T^2)\}.
\end{equation*}
This space plays the role of the solenoidal velocity space for the
Navier--Stokes equations. For a hydrostatically solenoidal field $u$, we denote
the associated vertical velocity by $w(u)$, using the vertical primitive
introduced above, so that $W=w(V)$. 

With $\Delta=\Delta_\H+\dz^2$, we define the Hilbert-space realization of the
hydrostatic Stokes operator $A$ by
\begin{equation*}
Au=-P\Delta u
\ \text{ and } \
D(A)=\{u\in\rH^2(\Omega)^2\mid Pu=u,\ \dz u|_{z=-h,0}=0\}.
\end{equation*}
In this geometry $P$ commutes with the Neumann Laplacian, so $A$ is nonnegative and
self-adjoint on $\rL_{\bar\sigma}^2(\Omega)$, see~\cite{GGHHK-17}.
Using the $\rL^2$ pivot, we regard $D(A)'$ as an extrapolation space and define
$A_{-1}$ on this space with domain $D(A_{-1})=\rL_{\bar\sigma}^2(\Omega)$.
We specify its action and define the boundary coefficients
$B_j\in D(A)'$ through the identities
\begin{equation*}
\langle A_{-1}u,\varphi\rangle=(u,A\varphi)_2
\ \text{ and } \
\langle B_j,\varphi\rangle=\int_{\T^2}g_j\cdot\varphi(\cdot,0)\d x_\H,
\end{equation*}
valid for every $u\in\rL_{\bar\sigma}^2(\Omega)$ and $\varphi\in D(A)$.
The functionals $B_j$ represent the wind stress through the adjoint of the upper
boundary value trace. The linear stochastic problem associated with \eqref{eq-stochastic-pe} subject to
\eqref{eq-stochastic-boundary} is given by
\begin{equation}\label{eq-stochastic-stokes}
\d Z+A_{-1}Z\d t=\sum_{j=1}^N B_j\d\beta_j
\ \text{ and } \ Z(0)=0.
\end{equation}
The Neumann boundary conditions are imposed through the weak formulation obtained
by integrating the Laplacian by parts. More precisely, for every
$\varphi\in D(A)$, it holds almost surely that
\begin{equation*}
\big(Z(t),\varphi\big)_{\rL^2(\Omega)}
+\int_0^t\big(Z(r),A\varphi\big)_{\rL^2(\Omega)}\,\d r
=\sum_{j=1}^N\beta_j(t)\int_{\T^2}
g_j(x_\H)\cdot\varphi(x_\H,0)\,\d x_\H,
\quad 0\le t\le T.
\end{equation*}
The boundary term on the right represents the prescribed stochastic stress at
$z=0$, while the homogeneous Neumann condition at $z=-h$ produces no boundary
term. Thus the normal derivative of $Z$ is prescribed in this weak sense,
without assuming that it has a classical boundary trace.
We solve \eqref{eq-stochastic-stokes} explicitly by defining the Neumann boundary
heat kernel $K$ by
\begin{equation*}
K_t(z)=\frac1h+\frac2h\sum_{n=1}^\infty
e^{-(n\pi/h)^2t}\cos(n\pi z/h),
\quad 0<t\le T,
\end{equation*}
and setting
\begin{equation}\label{eq-kernel-convolution}
Z(t)=\sum_{j=1}^N\int_0^t e^{-(t-r)A_{-1}}B_j\,\d\beta_j(r)
=\sum_{j=1}^N\int_0^t
P\big(K_{t-r}e^{(t-r)\Delta_\H}g_j\big)\,\d\beta_j(r),
\end{equation}
where $e^{-tA_{-1}}$ is the consistent extension of $e^{-tA}$ to $D(A)'$.
We call $Z$ the
stochastic convolution. The integral construction in
\autoref{prop-wind-path-regularity}, together with \eqref{eq-lift-weak},
shows that it satisfies \eqref{eq-stochastic-stokes} in the weak sense above.
Its zero initial value leaves the full initial datum $V_0$ in the remainder.
To subtract this linear response, we define the convection and its projection by
\begin{equation*}
B(u,v)=u\cdot\nablaH v+w(u)\dz v
\ \text{ and } \ F(u,v)=PB(u,v).
\end{equation*}
For smooth fields with zero normal velocity, incompressibility gives the
conservative form
\begin{equation*}
B(u,v)=\sum_{j=1}^2\partial_{x_j}(u_jv)+\dz(w(u)v),
\end{equation*}
which also defines the products by duality at lower regularity.
Following the Da Prato--Debussche method~\cite{DPD-03}, as used for this boundary
problem in~\cite[Section 5.1]{BHHS-24}, we define $v=V-Z$. Since $Z$ carries the
wind stress, the remainder has homogeneous Neumann conditions, and subtracting
\eqref{eq-stochastic-stokes} from \eqref{eq-stochastic-pe} subject to
\eqref{eq-stochastic-boundary} yields
\begin{equation}\label{eq-remainder}
\begin{aligned}
\dt v+A_{-1}v&=-F(v+Z,v+Z),\quad Pv=v,
\quad v(0)=V_0 \ \text{ and } \
w(v)|_{z=-h,0}= \dz v|_{z=-h,0}=0.
\end{aligned}
\end{equation}
The evolution equation is understood in $D(A)'$, and the homogeneous Neumann
condition is encoded weakly by $A_{-1}$. When $v\in D(A)$, it holds in the
Sobolev trace sense and $A_{-1}v=Av$.
For each fixed sample path, this is a deterministic evolution equation whose
coefficients depend on $Z$. The mixed terms $F(v,Z)$ and $F(Z,v)$ are linear in
$v$, while $F(Z,Z)$ is a prescribed forcing term. The estimates below justify
these products without renormalization.

To state the regularity near time zero, we use the notation
$\rH_z^s\rH_\H^m$ for the vertical Neumann Sobolev scale with values in the
horizontal periodic Sobolev space. The vertical scale is defined by cosine
expansion, or equivalently even reflection across the walls, with negative orders
interpreted by duality. In particular, $\rH_z^{-1}=(\rH_z^1)'$.
A missing exponent is zero, spatial norms without a domain are taken on $\Omega$,
and norms on $\T^2$ are indicated explicitly.
We use the same symbol $P$ for its consistent bounded extensions to these
anisotropic Sobolev spaces. To formulate the local problem, we define its
ground space $\rX_0$ by
\begin{equation*}
\rX_0=P\rH_z^{-1}\rH_\H^2.
\end{equation*}
The anisotropic realization $\sA$ on $\rX_0$ is defined by
\begin{equation*}
\sA u=A_{-1}u
\ \text{ and } \ \rX_1=D(\sA)=P(\rH_z^1\rH_\H^2\cap\rH_z^{-1}\rH_\H^4).
\end{equation*}
We equip $\rX_1$ with the graph norm of $\sA$. The operator identity is
understood in $D(A)'$, into which $\rX_0$ embeds, and its right-hand side
takes values in $\rX_0$ for $u\in\rX_1$.
The Neumann condition is encoded in this realization, since elements of
$\rX_1$ need not have a trace of their first vertical derivative.
The operators $A$, $A_{-1}$, and $\sA$ agree on intersections of their domains,
and their semigroups agree on intersections of the underlying spaces.
To specify the initial trace of the local maximal regularity solution,
we define $\rX_\gamma$ by
\begin{equation*}
\rX_\gamma=(\rX_0,\rX_1)_{1/2,2}
=P\big(\rL_z^2\rH_\H^2\cap\rH_z^{-1}\rH_\H^3\big).
\end{equation*}
In particular, $P\rL_z^2\rH_\H^3\hookrightarrow\rX_\gamma$.
For $v$ in the anisotropic maximal regularity class, the
evolution equation in \eqref{eq-remainder} takes the form
\begin{equation}\label{eq-remainder-anisotropic}
\begin{aligned}
\dt v+\sA v&=-F(v+Z,v+Z),\quad Pv=v
,\quad v(0)=V_0 \ \text{ and } \
w(v)|_{z=-h,0} =  \dz v|_{z=-h,0}=0.
\end{aligned}
\end{equation}
Here the homogeneous Neumann condition is understood weakly through $\sA$.
A pathwise solution of \eqref{eq-stochastic-pe} subject to
\eqref{eq-stochastic-boundary} is an adapted velocity $V=v+Z$ whose remainder
solves \eqref{eq-remainder} and whose boundary stress is understood through
\eqref{eq-lift-weak}. Its initial-time and positive-time regularity are specified
in the following theorem.

\begin{thm}\label{thm-global-wind}
Let $T>0$ be arbitrary, let $V_0\in P\rL_z^2\rH_\H^3$ be deterministic
and let $g_1,\ldots,g_N\in\rCinfty(\T^2)^2$.
Then the system \eqref{eq-stochastic-pe} subject to the boundary conditions
\eqref{eq-stochastic-boundary} admits an adapted pathwise solution
$V=v+Z$ on $[0,T]$. There is a strictly positive stopping time
$\tau\le T$ such that, almost surely, it holds that
\begin{equation*}
v\in\rH^1(0,\tau,\rX_0)\cap\rL^2(0,\tau,\rX_1)
\cap\rC([0,\tau],\rX_\gamma).
\end{equation*}
At positive times the remainder is a strong solution, and for every
$0<\delta<T$ it holds that
\begin{equation*}
v\in\rH^1(\delta,T,\rL_{\bar\sigma}^2(\Omega))
\cap\rL^2(\delta,T,D(A))\cap\rC([\delta,T],P\rH^1(\Omega)^2).
\end{equation*}
On the full time interval, it also satisfies
\begin{equation*}
v\in\rC([0,T],\rL_{\bar\sigma}^2(\Omega))
\cap\rL^2(0,T,P\rH^1(\Omega)^2).
\end{equation*}
The solution is pathwise unique among adapted solutions with the same
initial datum and Brownian motions whose remainders have these
initial-time and positive-time regularity properties.
Solutions on different finite time intervals agree upon restriction,
and hence define a unique global solution.
\end{thm}

The remainder $v=V-Z$ is first constructed using $\sA$ in the pair $(\rX_0,D(\sA))$ and
then continued from a positive regular time in
$(\rL_{\bar\sigma}^2(\Omega),D(A))$ using $A$. The energy bound
$v\in\rL^2(0,T,\rH^1(\Omega)^2)$ connects these constructions through the
comparison argument. Classical strong maximal regularity from time zero would
require $V_0\in P\rH^1(\Omega)^2$, whereas the anisotropic initial-time class
allows the data stated in the theorem. Global persistence in the original pair
would require additional horizontal estimates, since $\rH^2(\Omega)^2$ does
not embed into $\rX_1$, and is not asserted here.

\section{The boundary forcing and its regularity}
\label{sec-wind-stochastic}

We realize the wind stress by an explicit spatial lifting, which identifies
the boundary distribution in the stochastic Stokes equation, and then
estimate the convolution and its weighted vertical derivative to obtain
the coefficient bounds needed for the remainder equation.
To lift the prescribed boundary stress, for $g\in\rCinfty(\T^2)^2$
we define $\sL g$ by
\begin{equation*}
\sL g=P\left(\frac{(z+h)^2}{2h}g\right)
=\frac{(z+h)^2}{2h}g-\frac h6(I-P_\H)g.
\end{equation*}
The projection correction is independent of $z$, so differentiation gives
\begin{equation*}
P\sL g=\sL g,\quad \dz\sL g|_{z=-h}=0
\ \text{ and } \ 
\dz\sL g|_{z=0}=g.
\end{equation*}
The scalar factor is $(z+1)^2/2$ when $h=1$, and direct differentiation
gives the interior contribution
\begin{equation*}
-P\Delta\sL g=-\Delta_\H\sL g-\frac1hP_\H g.
\end{equation*}
This contribution must be retained when the boundary condition is
rewritten as a forcing. For every $\varphi\in D(A)$, integration by parts
with the extrapolated operator introduced in the setting gives
\begin{equation}\label{eq-polynomial-boundary}
\langle A_{-1}\sL g+P\Delta\sL g,\varphi\rangle
=\sum_{\ell=1}^3\int_{\Omega}\partial_\ell\sL g\cdot\partial_\ell\varphi\d x
+\int_{\Omega}\Delta\sL g\cdot\varphi\d x
=\int_{\T^2}g\cdot\varphi(\cdot,0)\d x_\H.
\end{equation}
Thus the polynomial lifting represents profiles with arbitrary mean, and
\eqref{eq-polynomial-boundary} identifies the coefficients in
\eqref{eq-stochastic-stokes} through
\begin{equation*}
B_j=A_{-1}\sL g_j+P\Delta\sL g_j.
\end{equation*}
Their regularity follows directly from the boundary pairing. For
$0<\varepsilon<1/2$ and every smooth hydrostatically solenoidal test field
$\varphi$ satisfying the homogeneous Neumann conditions, horizontal Sobolev
duality and the trace theorem with values in $\rH_\H^{-3}$ give
\begin{equation*}
\begin{aligned}
|\langle B_j,\varphi\rangle|
&=|\int_{\T^2}g_j\cdot\varphi(\cdot,0)\d x_\H|
\le \|g_j\|_{\rH_\H^3}\|\varphi(\cdot,0)\|_{\rH_\H^{-3}}\le C_\varepsilon\|g_j\|_{\rH_\H^3}
\|\varphi\|_{\rH_z^{1/2+\varepsilon}\rH_\H^{-3}}.
\end{aligned}
\end{equation*}
By density and duality, we obtain
\begin{equation*}
B_j\in P\rH_z^{-1/2-\varepsilon}\rH_\H^3
\ \text{ and } \ 
\|B_j\|_{\rH_z^{-1/2-\varepsilon}\rH_\H^3}
\le C_\varepsilon\|g_j\|_{\rH_\H^3}.
\end{equation*}

To relate the lifting to the explicit convolution, we expand the boundary
functional in the horizontal Fourier and vertical cosine basis, which gives
\begin{equation}\label{eq-projected-kernel}
e^{-tA_{-1}}(A_{-1}\sL g+P\Delta\sL g)
=P(K_te^{t\Delta_\H}g)
=\bigg(K_tI-\frac1h(I-P_\H)\bigg)e^{t\Delta_\H}g.
\end{equation}
For $t>0$, analyticity and consistency of the semigroups also yield
\begin{equation*}
e^{-tA}\sL g\in D(A)
\ \text{ and } \ e^{-tA_{-1}}A_{-1}\sL g=Ae^{-tA}\sL g.
\end{equation*}
Substituting these identities into \eqref{eq-kernel-convolution} therefore gives
\begin{equation*}
Z(t)=\sum_{j=1}^N\int_0^t
\big(Ae^{-(t-r)A}\sL g_j+e^{-(t-r)A}P\Delta\sL g_j\big)\d\beta_j(r).
\end{equation*}
The estimates in \autoref{prop-wind-path-regularity} justify the stochastic
integral near $r=t$. The second term cancels the interior contribution
of the lifting, leaving precisely the boundary forcing. The process
$\sum_j\sL g_j\beta_j(t)$ itself has upper derivative
$\sum_jg_j\beta_j(t)$ and hence does not impose the prescribed stress
$\sum_jg_j\dot\beta_j(t)$.

\begin{rem}[Comparison with the stationary Neumann map]
For mean-zero $g$, the stationary Neumann map may be taken directly
from~\cite[Proposition 3.5]{BHHS-24}. Its interior equation is
$-P\Delta Ng=0$, so integration by parts as in
\eqref{eq-polynomial-boundary} gives
\begin{equation*}
A_{-1}Ng=A_{-1}\sL g+P\Delta\sL g.
\end{equation*}
The mean condition is necessary for that unshifted stationary problem
and is imposed on the noise in~\cite[equation (4.4)]{BHHS-24}.
The polynomial lifting thus recovers the boundary-forcing formulation
of Da Prato and Zabczyk~\cite[Chapter 13]{DPZ-96}, used
in~\cite[Section 4.2]{BHHS-24}, while also allowing profiles with nonzero mean.
\end{rem}

Our next task is to identify the regularity of the stochastic convolution $Z$.
Stochastic maximal regularity yields a gain of one half of a power of the
second-order operator in square-integrable time norms
by~\cite[Theorem 1.1]{NVW-12}. Applied on the corresponding extrapolation
scale, this gives, almost surely, that
\begin{equation*}
Z\in\rL^2(0,T,\rH_z^{1/2-\varepsilon}\rH_\H^3),
\quad 0<\varepsilon<1/2.
\end{equation*}
We prove below the corresponding continuity for vertical orders below
$1/2$ and additional regularity of $(-z)\dz Z$, using the distance from
the noisy upper boundary to compensate for the singularity of the vertical
derivative.

\begin{prop}\label{prop-wind-path-regularity}
Fix $1/4<s<1/2$. The convolution in \eqref{eq-kernel-convolution} has
an adapted modification with $Z(0)=0$ such that, for every finite $T$,
almost surely, it holds that
\begin{equation*}
Z\in\rC([0,T],\rH_z^s\rH_\H^3\cap\rL_z^6\rH_\H^3)
\ \text{ and } \ 
(-z)\dz Z\in\rC([0,T],\rL_z^6\rH_\H^3).
\end{equation*}
Every finite moment of these continuous path norms is finite. The paths
are H\"older continuous in the first Sobolev norm with every exponent
less than $1/4-s/2$, and in the two $\rL_z^6\rH_\H^3$ norms with every
exponent less than $1/12$. In particular, it holds that
\begin{equation*}
Z,\ (-z)\dz Z\in\rC([0,T],\rL_z^6\rW_\H^{1,\infty}).
\end{equation*}
\end{prop}

\begin{proof}
The proof is divided into two steps. We first estimate the boundary heat
kernel and its time derivative, then use these bounds to construct the
stochastic integrals and prove their continuity through higher-moment
increment estimates.

\begin{step}[Kernel estimates]
To estimate $K_t$ near the upper boundary, we first rewrite its cosine
series as a sum of Gaussians. For every $n\in\mathbb Z$, the $n$-th
Fourier coefficient of the $2h$-periodic sum satisfies
\begin{equation*}
\frac1{2h\sqrt{\pi t}}\int_{-h}^{h}
\sum_{m\in\mathbb Z}e^{-(z-2mh)^2/(4t)}e^{-in\pi z/h}\d z
=\frac1{2h\sqrt{\pi t}}\int_{\R}e^{-y^2/(4t)}e^{-in\pi y/h}\d y
=\frac1h e^{-t(n\pi/h)^2}.
\end{equation*}
Pairing the coefficients with indices $n$ and $-n$ recovers the cosine
series defining $K_t$, and therefore gives
\begin{equation*}
K_t(z)=\frac1{\sqrt{\pi t}}\sum_{m\in\mathbb Z}
e^{-(z-2mh)^2/(4t)}.
\end{equation*}
For the term $m=0$, the substitution $z=\sqrt t\,y$ yields
\begin{equation*}
\|\frac1{\sqrt{\pi t}}e^{-z^2/(4t)}\|_{\rL_z^6}
=\frac{t^{-5/12}}{\sqrt\pi}
\bigg(\int_{-h/\sqrt t}^{0}e^{-3y^2/2}\d y\bigg)^{1/6}
\le\frac{t^{-5/12}}{\sqrt\pi}
\bigg(\int_{-\infty}^{0}e^{-3y^2/2}\d y\bigg)^{1/6}
\le Ct^{-5/12}.
\end{equation*}
Applying $(-z)\dz$ or $t\dt$ multiplies this Gaussian by $y^2/2$
or $-1/2+y^2/4$, respectively. Their composition also produces a
fixed polynomial in $y$. The same substitution therefore gives
$Ct^{-5/12}$ for each resulting norm, since every polynomial is
integrable against the Gaussian weight.
For $m\ne0$ and $-h\le z\le0$, it holds that
\begin{equation*}
|z-2mh|\ge(2|m|-1)h.
\end{equation*}
The remaining terms and their indicated derivatives are consequently
bounded uniformly for $0<t\le T$ by summing the Gaussian tails.
The cosine series also gives the Sobolev estimates by summing
$(1+n^2)^se^{-2t(n\pi/h)^2}$ and its time-differentiated counterpart.
Combining these calculations yields, for $0<t\le T$, that
\begin{equation*}
\begin{aligned}
\|K_t\|_{\rL_z^6}+\|(-z)\dz K_t\|_{\rL_z^6}
+t\|\dt K_t\|_{\rL_z^6}+t\|(-z)\dz\dt K_t\|_{\rL_z^6}
&\le C_Tt^{-5/12},\\
\|K_t\|_{\rH_z^s}+t\|\dt K_t\|_{\rH_z^s}
&\le C_Tt^{-1/4-s/2}.
\end{aligned}
\end{equation*}
The horizontal heat semigroup is bounded on $\rH_\H^3$, its time
derivative costs at most $Ct^{-1}$ in that norm, and the projection
correction in \eqref{eq-projected-kernel} is constant in $z$ and
bounded on $\rH_\H^3$. Hence the same powers of $t$ control each
integrand in the assertion, including its time derivative.
\end{step}

\begin{step}[Stochastic integration and continuity]
The squared kernel bounds are integrable at zero because $s<1/2$.
More explicitly, for deterministic elementary integrands $\Phi_j$ and
$r\ge6$, Minkowski's inequality and the Hilbert-valued
Burkholder--Davis--Gundy inequality \cite[Theorem~4.4]{NVW-07} give
\begin{equation*}
\begin{aligned}
\bigg(\EE\|\sum_{j=1}^N\int_0^T \Phi_j(\rho)\d\beta_j(\rho)
\|_{\rL_z^6\rH_\H^3}^{r}\bigg)^{1/r}
&\le
\|\bigg(\EE\|\sum_{j=1}^N\int_0^T
\Phi_j(\rho,z)\d\beta_j(\rho)\|_{\rH_\H^3}^{r}
\bigg)^{1/r}\|_{\rL_z^6}\\
&\le C_r\|\bigg(\sum_{j=1}^N\int_0^T
\|\Phi_j(\rho,z)\|_{\rH_\H^3}^2\d\rho\bigg)^{1/2}
\|_{\rL_z^6}\\
&\le C_r\bigg(\sum_{j=1}^N\int_0^T
\|\Phi_j(\rho)\|_{\rL_z^6\rH_\H^3}^2\d\rho\bigg)^{1/2}.
\end{aligned}
\end{equation*}
The first inequality uses $r\ge6$, the second applies the Hilbert-valued
inequality at each $z$, and the last is another application of
Minkowski's inequality. Direct application of the Hilbert-valued inequality
gives the same final bound with $\rL_z^6\rH_\H^3$ replaced by
$\rH_z^s\rH_\H^3$. Approximation by elementary integrands therefore
constructs the convolution and its weighted derivative in the three spatial
norms, since the kernel estimates make these right-hand sides finite.

To obtain continuity in these same spatial norms, we now estimate
temporal increments in arbitrarily high moments. This uses the
deterministic time-independent profiles and the strict inequality
$s<1/2$, and is additional to the square-integrable time estimate.

For $0<\delta\le T-t$, we split the increment of the stochastic
integral into the contribution from the new time interval and the
difference of the kernels on the old interval. On the old interval
with lag less than $\delta$, we estimate both kernels separately by
their size bounds. For lag at least $\delta$, we use the time
derivative bound and the fundamental theorem of calculus.
If the kernel exponent is $a<1/2$, the new interval and the small
old lags each contribute at most a constant times the first integral
below. We therefore obtain
\begin{equation*}
\int_0^\delta t^{-2a}\d t
+\delta^2\int_\delta^T t^{-2a-2}\d t
\le C_{T,a}\delta^{1-2a}.
\end{equation*}
Using this inequality with $a=1/4+s/2$ and $a=5/12$, respectively,
the preceding stochastic integral estimates yield, for every $r\ge6$, that
\begin{equation*}
\begin{aligned}
\bigg(\EE\|Z(t+\delta)-Z(t)\|_{\rH_z^s\rH_\H^3}^r\bigg)^{1/r}
&\le C_{T,r}\delta^{1/4-s/2},\\
\bigg(\EE\bigg[
\|Z(t+\delta)-Z(t)\|_{\rL_z^6\rH_\H^3}
+\|(-z)\dz(Z(t+\delta)-Z(t))\|_{\rL_z^6\rH_\H^3}
\bigg]^r\bigg)^{1/r}
&\le C_{T,r}\delta^{1/12}.
\end{aligned}
\end{equation*}
The constants depend on $\sum_j\|g_j\|_{\rH_\H^3}^2$.
For fixed $r$, the quantitative Kolmogorov continuity theorem gives
H\"older exponents less than $1/4-s/2-1/r$ in the Sobolev norm
and less than $1/12-1/r$ in the two $\rL_z^6\rH_\H^3$ norms,
provided these upper bounds are positive. Since $r$ can be chosen
arbitrarily large, this proves the asserted H\"older regularity and,
in particular, continuity in all three spatial norms. Thus the
continuous-path conclusion follows from these higher-moment increment
bounds, rather than from the $\rL^2$-in-time estimate alone.

For any prescribed finite moment of the path norms, we choose $r$
larger than that moment and large enough for the preceding
Kolmogorov estimates. Their quantitative form, together with the
zero initial value, gives the required path-norm moments.
Lower moments follow from H\"older's inequality.

Testing on compact vertical subintervals and applying stochastic Fubini
identifies the weighted integral with the interior distribution
$(-z)\dz Z$, so the modifications agree in their common spaces and
remain adapted under the usual conditions on the filtration.
Finally, the embedding $\rH_\H^3\hookrightarrow\rW_\H^{1,\infty}$
gives the last assertion.\qedhere
\end{step}
\end{proof}
To justify the later energy tests, we also need spatially smooth
approximations whose weighted norms remain controlled.

\begin{prop}\label{prop-wind-regularization}
For $\varepsilon>0$, we define the spatial approximation by
$Z^\varepsilon=e^{-\varepsilon A}Z$. It has continuous paths with
arbitrary spatial Sobolev regularity and homogeneous Neumann conditions.
For every finite $T$ and $1\le r<\infty$, it holds that
\begin{equation*}
\EE\big[
\|Z^\varepsilon-Z\|_{\rC([0,T],\rH_z^s\rH_\H^3)}
+\|Z^\varepsilon-Z\|_{\rC([0,T],\rL_z^6\rH_\H^3)}
+\|(-z)\dz(Z^\varepsilon-Z)\|_{\rC([0,T],\rL_z^6\rH_\H^3)}
\big]^r\longrightarrow0.
\end{equation*}
The corresponding norms of $Z^\varepsilon$ have uniformly bounded
moments for $0<\varepsilon\le1$. Along a deterministic sequence
$\varepsilon_n\downarrow0$, the convergence holds almost surely
on every finite time interval.
\end{prop}

\begin{proof}
We use the shifted kernels to prove convergence at fixed times and the
preceding H\"older estimates to obtain convergence of the continuous paths.
The kernel formula shows that $Z^\varepsilon$ is obtained by replacing
each kernel time $t-r$ in \eqref{eq-kernel-convolution} by
$t-r+\varepsilon$. In particular, its weighted derivative is obtained
from $(-z)\dz K_{t-r+\varepsilon}$ directly. Since
$(t+\varepsilon)^{-a}\le t^{-a}$, the proof of
\autoref{prop-wind-path-regularity} gives uniform bounds for a positive
H\"older norm in each of the three spaces.

At every positive kernel time, the shifted kernels converge to the
original kernels, and their squared norms are bounded by an integrable
multiple of $t^{-2a}$, with $a<1/2$. Dominated convergence and the
stochastic integral estimates in \autoref{prop-wind-path-regularity}
therefore give convergence at every fixed time and hence uniformly
on any finite time grid in every finite probability moment. The uniform
H\"older bound controls the error between adjacent grid points, so letting
the mesh tend to zero gives
the stated convergence in continuous path norms. Markov's inequality
and Borel--Cantelli then select a single deterministic subsequence
converging almost surely on all integer horizons. The exponential
factors in the shifted Fourier--cosine series give spatial smoothness
and the homogeneous Neumann conditions. \qedhere
\end{proof}

For a smooth hydrostatically solenoidal Neumann test field $\varphi$,
stochastic Fubini and \eqref{eq-polynomial-boundary} give
\begin{equation}\label{eq-lift-weak}
(Z(t),\varphi)_2+\int_0^t(Z(r),A\varphi)_2\d r
=\sum_{j=1}^N\beta_j(t)\int_{\T^2}g_j\cdot\varphi(\cdot,0)\d x_\H.
\end{equation}
For $Z^\varepsilon$, the right-hand side contains
$e^{-\varepsilon A}\varphi$ in place of $\varphi$. The displayed
identity follows by passing to the limit with
\autoref{prop-wind-regularization} and the convergence
$e^{-\varepsilon A}\varphi\to\varphi$ in $D(A)$.
This gives the weak formulation of \eqref{eq-stochastic-stokes}, which,
after adding the remainder equation, yields the weak formulation of
\eqref{eq-stochastic-pe} subject to \eqref{eq-stochastic-boundary}.
For profiles with nonzero horizontal divergence, the eliminated pressure
may itself be a distribution in time, while the projected identity
\eqref{eq-lift-weak} remains unambiguous.

Having identified the boundary forcing in the Da Prato--Debussche
decomposition, we next estimate the nonlinear terms in the remainder equation.
\section{Local existence and comparison}
\label{sec-wind-local}

We improve the nonlinear estimates of
\cite[Lemmas~5.4 and~5.5]{BHHS-24} in the horizontally stronger
spaces $\rX_0$ and $\rX_1$, where the additional horizontal
derivatives allow the nonlinear terms to be estimated in square-integrable
time and space norms.

For the local construction, we define the time-dependent ground space and
the maximal regularity space by
\begin{equation*}
\mathbb E_0(0,T)=\rL^2(0,T,\rX_0)
\ \text{ and }\
\mathbb E_1(0,T)=\rH^1(0,T,\rX_0)\cap\rL^2(0,T,\rX_1).
\end{equation*}
\begin{prop}[Local existence]
\label{prop-wind-local-hilbert}
Let $1/4<s<1/2$ and let
$Z\in\rC([0,T],P\rH_z^s\rH_\H^3)$.
For every $v_0\in\rX_\gamma$, there is a maximal existence time
$t_+\in(0,T]$ and a unique solution of \eqref{eq-remainder-anisotropic} on
$[0,t_+)$ with initial value $v_0$. For every $0<S<t_+$, it holds that
\begin{equation*}
v\in\rH^1(0,S,\rX_0)\cap\rL^2(0,S,\rX_1)
\cap\rC([0,S],\rX_\gamma).
\end{equation*}
Maximality is understood in this anisotropic maximal regularity
class. If $t_+<T$, it holds that
\begin{equation*}
\|v\|_{\rH^1(0,t_+,\rX_0)}
+\|v\|_{\rL^2(0,t_+,\rX_1)}=\infty.
\end{equation*}
The solution depends continuously on $v_0$ and $Z$ in the indicated
spaces on a common sufficiently short interval. In particular,
arbitrary $v_0\in P\rL_z^2\rH_\H^3$ are admissible.
\end{prop}

\begin{proof}
The proof is divided into three steps. We first establish the spatial
estimates, then apply critical maximal regularity theory.
\begin{step}[The linear theory and spatial estimates]
The operator $\sA$ on $\rX_0$ with domain $\rX_1=D(\sA)$ has maximal $\rL^2$ regularity by
\cite[Proposition~3.4]{BHHS-24}, applied with vertical order $-1$
and horizontal order $2$. The same proposition identifies its trace
space as $(\rX_0,\rX_1)_{1/2,2}=\rX_\gamma$.
For use in the nonlinear estimates, we define the fractional spaces by
$\rX_\beta=D((1+\sA)^\beta)$ for $0<\beta<1$.
The characterization in~\cite[Proposition~3.4]{BHHS-24} gives, with
equivalent norms, that
\begin{equation*}
\rX_\beta=[\rX_0,\rX_1]_\beta
=P\big(\rH_z^{2\beta-1}\rH_\H^2
\cap\rH_z^{-1}\rH_\H^{2+2\beta}\big).
\end{equation*}
Complex interpolation between the two anisotropic Sobolev spaces
in this intersection, followed by the Sobolev embeddings, yields
\begin{equation*}
\rX_\beta\hookrightarrow\rH_z^a\rH_\H^{2+b},
\quad a\ge-1,\quad b\ge0
\ \text{ and } \ a+1+b\le2\beta.
\end{equation*}

For smooth fields, the divergence constraint rewrites the
nonlinearity as $F=F_\H+F_z$, where we define its two parts by
\begin{equation*}
F_\H(u,v)=P\sum_{j=1}^2\partial_{x_j}(u_jv)
\ \text{ and } \ F_z(u,v)=P\dz(w(u)v).
\end{equation*}
First,  using
Cauchy--Schwarz and the identity $\dz w(u)=-\divH u$ yield
\begin{equation*}
\|w(u)\|_{\rH_z^1\rH_\H^2}
\le C\|\divH u\|_{\rL_z^2\rH_\H^2}
\le C\|u\|_{\rL_z^2\rH_\H^3}.
\end{equation*}
To obtain the corresponding estimate at negative vertical order,
we use duality. For every smooth scalar test function $\varphi$,
one has
\begin{equation*}
\begin{aligned}
|\langle w(u),\varphi\rangle|
=|\langle\divH u,
\int_{\cdot}^0\varphi(\zeta)\d\zeta\rangle|
\le \|\divH u\|_{\rH_z^{-1}\rH_\H^2}
\|\int_{\cdot}^0\varphi(\zeta)\d\zeta\|_{\rH_z^1\rH_\H^{-2}}
\le C\|u\|_{\rH_z^{-1}\rH_\H^3}
\|\varphi\|_{\rL_z^2\rH_\H^{-2}}.
\end{aligned}
\end{equation*}
Complex interpolation of these two bounds with parameter $1/4$
therefore gives
\begin{equation*}
\|w(u)\|_{\rH_z^{1/4}\rH_\H^2}
\le C\|u\|_{\rH_z^{-3/4}\rH_\H^3}.
\end{equation*}
By density, this estimate extends the vertical primitive to the
stated negative-order space.
The embeddings $\rH_z^{1/4}\hookrightarrow\rL_z^4$ and the
horizontal algebra property give
\begin{equation}
\label{eq-wind-bilinear-vertical}
\|F_z(u,v)\|_{\rX_0}
\le C\|u\|_{\rH_z^{-3/4}\rH_\H^3}
\|v\|_{\rH_z^{1/4}\rH_\H^2}
\le C\|u\|_{\rX_{5/8}}\|v\|_{\rX_{5/8}}.
\end{equation}
For the horizontal part, we use the standard one-dimensional
multiplication estimate. For $0<a<1/2$, the bounded-domain
multiplication theorem \cite[Theorem~7.4]{BehzadanHolst-21} gives
$\|f\phi\|_{\rH_z^a}\le C\|f\|_{\rH_z^a}\|\phi\|_{\rH_z^1}$,
and duality therefore yields
\begin{equation}
\label{eq-wind-opposite-orders}
\|fg\|_{\rH_z^{-1}}\le C\|f\|_{\rH_z^a}\|g\|_{\rH_z^{-a}}.
\end{equation}
To apply \eqref{eq-wind-opposite-orders}
with horizontal Sobolev values, we use the bounded multiplication
$\rH_\H^2\times\rH_\H^{-2}\longrightarrow\rH_\H^{-2}$.
The one-dimensional embedding $\rH_z^1\hookrightarrow\rL_z^\infty$
and the product rule give the corresponding multiplier bounds at
vertical orders $0$ and $1$. Interpolation at order $a$ and the same
duality argument yield
$\|fg\|_{\rH_z^{-1}\rH_\H^2}
\le C\|f\|_{\rH_z^a\rH_\H^2}\|g\|_{\rH_z^{-a}\rH_\H^2}$.
Expanding the horizontal derivative and taking $a=1/4$, we obtain
\begin{equation}
\label{eq-wind-bilinear-horizontal}
\begin{aligned}
\|F_\H(u,v)\|_{\rX_0}
&\le C\|u\|_{\rH_z^{1/4}\rH_\H^2}
\|v\|_{\rH_z^{-1/4}\rH_\H^3}
+C\|v\|_{\rH_z^{1/4}\rH_\H^2}
\|u\|_{\rH_z^{-1/4}\rH_\H^3}\\
&\le C\|u\|_{\rX_{5/8}}\|v\|_{\rX_{7/8}}
+C\|v\|_{\rX_{5/8}}\|u\|_{\rX_{7/8}}.
\end{aligned}
\end{equation}
To include the prescribed path, we define $N_Z$ by
$N_Z=\sup_{t\le T}\|Z(t)\|_{\rH_z^s\rH_\H^3}$ and apply the same
estimates with one or both arguments replaced by $Z$, which yields
\begin{equation*}
\|F(v,Z)+F(Z,v)\|_{\rX_0}
\le CN_Z(\|v\|_{\rX_{5/8}}+\|v\|_{\rX_{7/8}})
\ \text{ and } \ \|F(Z,Z)\|_{\rX_0}\le CN_Z^2.
\end{equation*}
Only the norms of $Z$ in $\rH_z^{1/4}\rH_\H^2$,
$\rH_z^{-1/4}\rH_\H^3$, and $\rH_z^{-3/4}\rH_\H^3$ occur here,
so all three are controlled by $N_Z$ without requiring
$Z\in\rX_{7/8}$.
\end{step}

\begin{step}[Application of the critical maximal regularity theory]
We apply the framework of~\cite[Theorem~1.2]{PW17} with time
integrability $p=2$, weight $\mu=1$, and $\beta=7/8$.
The linear theory established above gives maximal $\rL^2$ regularity.
Moreover, $\rX_0$ is a Hilbert space and $1+\sA$ has a bounded
$H^\infty$-calculus, so the structural condition~(S) follows
from~\cite[Remark~1.1(i)]{PW17}.

Since $\rX_\gamma=\rX_{1/2}$ with equivalent norms, interpolation gives
\begin{equation*}
\|u\|_{\rX_{5/8}}
\le C\|u\|_{\rX_\gamma}^{2/3}
\|u\|_{\rX_{7/8}}^{1/3}.
\end{equation*}
Consequently, for
$\|u\|_{\rX_\gamma}+\|v\|_{\rX_\gamma}\le R$,
bilinearity and \eqref{eq-wind-bilinear-vertical} together with
\eqref{eq-wind-bilinear-horizontal} yield
\begin{equation*}
\begin{aligned}
\|F(u,u)-F(v,v)\|_{\rX_0}
\le C_R\big[
\big(1+\|u\|_{\rX_{7/8}}^{1/3}
+\|v\|_{\rX_{7/8}}^{1/3}\big)
\|u-v\|_{\rX_{7/8}}+\big(1+\|u\|_{\rX_{7/8}}
+\|v\|_{\rX_{7/8}}\big)
\|u-v\|_{\rX_{5/8}}\big].
\end{aligned}
\end{equation*}
Thus hypothesis~(H2) holds with
$
(\rho_1,\beta_1)=(1/3,7/8)
$ \text{ and } $
(\rho_2,\beta_2)=(1,5/8)$
For our choices of $p$ and $\mu$, hypothesis~(H3) requires
$\rho_j(\beta-1/2)+\beta_j-1/2\le1/2$.
Both pairs satisfy this condition with equality, since it holds that
\begin{equation*}
\frac13\bigg(\frac78-\frac12\bigg)+\frac78-\frac12=\frac12
\ \text{ and } \
\bigg(\frac78-\frac12\bigg)+\frac58-\frac12=\frac12.
\end{equation*}
The operator is independent of the unknown, so hypothesis~(H1)
is immediate for the autonomous equation with nonlinearity $-F(v,v)$.

To include the prescribed path $Z$, the preceding spatial estimates give
\begin{equation*}
\|F(u-v,Z(t))+F(Z(t),u-v)\|_{\rX_0}
\le CN_Z\|u-v\|_{\rX_{7/8}}
\ \text{ and } \
F(Z,Z)\in\rC([0,T],\rX_0).
\end{equation*}
The mixed terms therefore correspond to $\rho=0$ and $\beta_j=7/8$,
which satisfy~(H3) strictly.
Although~\cite[Theorem~1.2]{PW17} is stated for autonomous equations,
its proof applies to these additional time-dependent terms, as
in~\cite[Section~5.3]{BHHS-24}. Indeed, for
$0<T_0\le\min\{1,T\}$ and
$u\in\mathbb E_1(0,T_0)$ with $u(0)=0$, the mixed derivative
embedding into $\rL^{8/3}(0,T_0,\rX_{7/8})$ and H\"older's
inequality give
\begin{equation*}
\begin{aligned}
\|F(u,Z)+F(Z,u)\|_{\mathbb E_0(0,T_0)}
&\le CN_ZT_0^{1/8}\|u\|_{\mathbb E_1(0,T_0)},\\
\|F(Z,Z)\|_{\mathbb E_0(0,T_0)}
&\le CN_Z^2T_0^{1/2}.
\end{aligned}
\end{equation*}
The embedding constant is uniform for zero initial trace.
These bounds supply the required smallness on short intervals
in the cited proof, without differentiating $Z$ in time.

The maximal regularity theorem therefore yields a unique
local solution of \eqref{eq-remainder-anisotropic} in
$\mathbb E_1(0,T_0)\cap\rC([0,T_0],\rX_\gamma)$
for every $v_0\in\rX_\gamma$, together with continuous dependence
on the initial value. Applying the same estimates to differences
of prescribed paths gives continuous dependence on $Z$.
\qedhere
\end{step}
\end{proof}

We now fix $T_0\in(0,t_+)$ and compare the local solution on
$[0,T_0]$ with solutions in the classical strong class. The term
$w(u)\dz Z$ is estimated using the weighted derivative from
\autoref{prop-wind-path-regularity} and the quotient $w(u)/(-z)$.
For this purpose, we define the Hardy average $\mathcal H$ and
write the weighted derivative as $R$ by setting
\begin{equation*}
\mathcal H f(z)=\frac1{-z}\int_0^z f(\zeta)\d\zeta
\ \text{ and } \ R=(-z)\dz Z.
\end{equation*}
For $1<\ell<\infty$ and every Banach space $\rE$, the identity
$\mathcal H f(z)=-\int_0^1f(\theta z)\d\theta$ and
Minkowski's inequality yield
\begin{equation}
\label{eq-wind-hardy}
\|\mathcal H f\|_{\rL_z^\ell(\rE)}
\le\int_0^1\theta^{-1/\ell}\d\theta\,
\|f\|_{\rL_z^\ell(\rE)}
=\frac\ell{\ell-1}\|f\|_{\rL_z^\ell(\rE)},
\end{equation}
In particular,
$w(u)/(-z)=-\mathcal H(\divH u)$ for every hydrostatically solenoidal
field. All subsequent local assertions assume
$Z,R\in\rC([0,T],\rL_z^6\rW_\H^{1,\infty})$, as provided by
\autoref{prop-wind-path-regularity}. To record their size, we define $K$ by
\begin{equation*}
K(t)=\|Z(t)\|_{\rL_z^6\rW_\H^{1,\infty}}
+\|R(t)\|_{\rL_z^6\rW_\H^{1,\infty}}.
\end{equation*}

\begin{lem}[Energy and comparison]
\label{lem-wind-relative-energy}
The solution of \autoref{prop-wind-local-hilbert} satisfies
\begin{equation*}
v\in\rC([0,T_0],\rL^2)\cap\rL^2(0,T_0,\rH^1)
\ \text{ and } \ \dt v\in\rL^2(0,T_0,(\rH^1)').
\end{equation*}
It obeys the remainder energy equality and the estimate
\begin{equation*}
\frac{\d}{\d t}\|v\|_2^2+\|\nabla v\|_2^2
\le C(1+K^4)\|v\|_2^2+CK^4.
\end{equation*}
The same energy equality and estimate hold on $[0,T]$ for every
solution $v$ of \eqref{eq-remainder} satisfying
$v\in\rH^1(0,T,\rL_{\bar\sigma}^2(\Omega))
\cap\rL^2(0,T,D(A))$.
For the same prescribed path $Z$, any two such solutions with
the same initial value coincide on $[0,T]$.
This uniqueness statement also holds when either or both solutions
belong to the local maximal regularity class of
\autoref{prop-wind-local-hilbert} on $(0,T)$.
The same conclusions apply after translating the initial time.
\end{lem}

\begin{proof}
The proof is divided into two steps. We first justify the energy
identity and estimate the terms containing $Z$, then apply the same
argument to the difference of two solutions.
\begin{step}[The noise terms and the energy identity]
The one-dimensional Gagliardo--Nirenberg inequality gives
$\|u\|_{\rL_z^3\rL_\H^2}
\le C\|u\|_2^{5/6}(\|u\|_2+\|\dz u\|_2)^{1/6}$.
Together with \eqref{eq-wind-hardy}, this yields
\begin{equation*}
\begin{aligned}
|\langle B(u,Z),u\rangle|
&\le CK\|\nablaH u\|_2\|u\|_2^{5/6}
(\|u\|_2+\|\dz u\|_2)^{1/6}
+CK\|u\|_{\rL_z^3\rL_\H^2}^2,\\
\|F(Z,Z)\|_2
&\le C\big(\|Z\|_{\rL_z^6\rL_\H^\infty}
\|\nablaH Z\|_{\rL_z^6\rL_\H^\infty}
+\|R\|_{\rL_z^6\rL_\H^\infty}
\|\mathcal H(\divH Z)\|_{\rL_z^6\rL_\H^\infty}\big)
\le CK^2.
\end{aligned}
\end{equation*}
The second estimate uses $\rL_z^3\rL_\H^\infty\hookrightarrow
\rL^2(\Omega)$. For every $\eta>0$, Young's inequality consequently yields
\begin{equation}
\label{eq-wind-Z-relative}
|\langle F(u,Z),u\rangle|
\le\eta\|\nabla u\|_2^2+C_\eta(1+K^4)\|u\|_2^2
\ \text{ and } \ \|F(Z,Z)\|_2\le CK^2.
\end{equation}
Transport by $(Z,w(Z))$ is skew symmetric, so
$\langle F(Z,u),u\rangle=0$. This identity follows first for
smooth fields from incompressibility and the zero wall values of
$w(Z)$, and then for $Z$ by \autoref{prop-wind-regularization}, since the
preceding bounds and the Hardy identity justify passage to the limit.

The embeddings $\rX_1\hookrightarrow\rH^1$,
$\rX_\gamma\hookrightarrow\rL^2$, and $\rX_0\hookrightarrow(\rH^1)'$
give the asserted regularity. The Hilbert triple chain rule,
the usual self-advection cancellation, and the preceding noise
cancellation therefore yield
\begin{equation*}
\frac12\frac{\d}{\d t}\|v\|_2^2+\|\nabla v\|_2^2
=-\langle F(v,Z),v\rangle-\langle F(Z,Z),v\rangle.
\end{equation*}
The displayed estimate in the lemma now follows from
\eqref{eq-wind-Z-relative} and Young's inequality. For a local strong
remainder the same Hilbert triple chain rule and cancellations apply,
so this argument also proves its energy equality and estimate.
\end{step}

\begin{step}[Comparison of the regularity classes]
For $U=v_1-v_2$, the same chain rule and cancellations give
\begin{equation*}
\frac12\frac{\d}{\d t}\|U\|_2^2+\|\nabla U\|_2^2
=-\langle F(U,v_2),U\rangle-\langle F(U,Z),U\rangle.
\end{equation*}
If $v_2$ belongs to the class in \autoref{prop-wind-local-hilbert},
we define the two coefficient functions by
\begin{equation*}
a=\|\nablaH v_2\|_{\rL_z^\infty\rL_\H^2}
\ \text{ and } \ b=\|\dz v_2\|_{\rL_z^2\rL_\H^\infty}.
\end{equation*}
Both belong to $\rL^2(I)$. The horizontal Ladyzhenskaya inequality
and $\|w(U)\|_{\rL_z^\infty\rL_\H^2}
\le C\|\nablaH U\|_2$ yield
\begin{equation*}
|\langle F(U,v_2),U\rangle|
\le Ca\|U\|_2(\|U\|_2+\|\nablaH U\|_2)
+Cb\|\nablaH U\|_2\|U\|_2.
\end{equation*}
If $v_2$ is in the strong class, its trace and interpolation instead
give $v_2\in\rC(I,\rH^1)\cap\rL^4(I,\rH^{3/2})$.
In this case we use
$a=\|\nablaH v_2\|_{\rL^3}$ and
$b=\|\dz v_2\|_{\rL_z^2\rL_\H^4}$, which belong to
$\rL^4(I)$. The three-dimensional interpolation inequality for
$\|U\|_3^2$ and the horizontal one for
$\|U\|_{\rL_z^2\rL_\H^4}$ give
\begin{equation*}
|\langle F(U,v_2),U\rangle|
\le Ca\|U\|_2(\|U\|_2+\|\nabla U\|_2)
+Cb\|\nablaH U\|_2\|U\|_2^{1/2}
(\|U\|_2+\|\nablaH U\|_2)^{1/2}.
\end{equation*}
Young's inequality and \eqref{eq-wind-Z-relative} leave an
integrable Gronwall coefficient, bounded respectively by
$C(a+a^2+b^2+1+K^4)$ and
$C(a^2+b^2+b^4+1+K^4)$. Since the initial difference vanishes,
Gronwall's inequality proves the assertion.
\qedhere
\end{step}
\end{proof}

For the continuation from a positive regular time, we use the realization $A$
on $\rL_{\bar\sigma}^2(\Omega)$. For $v$ in the corresponding strong class,
the evolution equation in \eqref{eq-remainder} takes the form
\begin{equation}\label{eq-remainder-strong}
\dt v+Av=-F(v+Z,v+Z).
\end{equation}
Here membership in $D(A)$ imposes the homogeneous Neumann boundary conditions
displayed in \eqref{eq-remainder}.

\begin{prop}[Local strong remainder solutions and continuation]
\label{prop-wind-local-strong}
For every $v_0\in P\rH^1(\Omega)^2$ there is a unique local solution
of \eqref{eq-remainder-strong} with initial value $v_0$ such that
\begin{equation*}
v\in\rH^1(0,T_0,\rL_{\bar\sigma}^2)\cap\rL^2(0,T_0,D(A))
\cap\rC([0,T_0],P\rH^1).
\end{equation*}
Solutions depend continuously on the initial value in $P\rH^1$
and on $(Z,(-z)\dz Z)$ in
$\rC([0,T],\rL_z^6\rW_\H^{1,\infty})^2$.
Convergent approximations exist and converge in these strong solution
norms on every compact interval of existence of the limiting solution.
If the maximal lifetime $T_*<T$ is finite, then it holds that
\begin{equation*}
\sup_{t<T_*}\|v(t)\|_{\rH^1}^2
+\int_0^{T_*}\|v(t)\|_{\rH^2}^2\d t=\infty.
\end{equation*}
\end{prop}

\begin{proof}
Local existence and continuous dependence follow from the critical
maximal regularity argument in the proof of
\autoref{prop-wind-local-hilbert}, now applied to
$(\rL_{\bar\sigma}^2(\Omega),D(A))$ with trace space
$P\rH^1(\Omega)^2$ and fractional exponent $3/4$.
Uniqueness follows from \autoref{lem-wind-relative-energy}.

To prove the continuation criterion, we use the standard bilinear
estimate
\begin{equation}
\label{eq-wind-strong-bilinear}
\|F(u,v)\|_2
\le C\|u\|_{\rH^{3/2}}\|v\|_{\rH^{3/2}}.
\end{equation}
For the mixed terms, \eqref{eq-wind-hardy} and the
one-dimensional interpolation inequality give
\begin{equation}
\label{eq-wind-strong-cross}
\|F(v,Z)+F(Z,v)\|_2
\le CK\big(\|v\|_{\rH^1}
+\|v\|_{\rH^1}^{5/6}\|v\|_{\rH^2}^{1/6}\big).
\end{equation}
Suppose that $T_*<T$ and that the quantity in the continuation
assertion is finite, so that it holds that
\begin{equation*}
\sup_{t<T_*}\|v(t)\|_{\rH^1}^2
+\int_0^{T_*}\|v(t)\|_{\rH^2}^2\d t<\infty.
\end{equation*}
Interpolation and \eqref{eq-wind-strong-bilinear} then yield
\begin{equation*}
\int_0^{T_*}\|F(v,v)\|_2^2\d t
\le C\sup_{t<T_*}\|v(t)\|_{\rH^1}^2
\int_0^{T_*}\|v(t)\|_{\rH^2}^2\d t<\infty.
\end{equation*}
Since $K$ is bounded on $[0,T]$, Young's inequality and
\eqref{eq-wind-strong-cross} also give
\begin{equation*}
\int_0^{T_*}\|F(v,Z)+F(Z,v)\|_2^2\d t
\le C\sup_{t<T_*}K(t)^2
\int_0^{T_*}\big(\|v(t)\|_{\rH^1}^2
+\|v(t)\|_{\rH^2}^2\big)\d t<\infty.
\end{equation*}
The quadratic forcing belongs to
$\rL^2(0,T_*,\rL_{\bar\sigma}^2(\Omega))$
by \eqref{eq-wind-Z-relative}.
Consequently, the equation gives
$\dt v\in\rL^2(0,T_*,\rL_{\bar\sigma}^2(\Omega))$.
Together with the assumed spatial bound, this places $v$ in the
strong maximal regularity class up to $T_*$.
The trace theorem therefore supplies a limit
$v(T_*)\in P\rH^1(\Omega)^2$.
Applying the local result at this value and using uniqueness
extends the solution beyond $T_*$, contradicting maximality.
\end{proof}
\section{Global a priori estimates}
\label{sec-wind-energy}

Following the global estimates of Cao and
Titi~\cite[Sections 3.1--3.3]{CT-07}, we retain their estimates for
the terms containing only $v$ and estimate the additional terms involving
$Z$ through $(-z)\dz Z$. This avoids requiring an unweighted vertical
derivative of the stochastic convolution.

\begin{prop}[Global a priori estimate]
\label{prop-wind-global-energy}
Let $T>0$ and let $Z$ be the stochastic convolution defined in
\eqref{eq-kernel-convolution}.
We work pathwise on the event of probability one on which the
conclusions of \autoref{prop-wind-path-regularity} and
\autoref{prop-wind-regularization} hold.
To measure the coefficients needed below, we define $M_T$ by
\begin{equation*}
M_T=1+\sup_{0\le t\le T}
\bigl(\|Z(t)\|_{\rL_z^6\rW_\H^{1,\infty}}
+\|(-z)\dz Z(t)\|_{\rL_z^6\rW_\H^{1,\infty}}\bigr).
\end{equation*}
Let $0<S\le T$ and let $v$ be a local strong solution of
\eqref{eq-remainder-strong} on $[0,S)$, with $v(0)=v_0\in\rH^1(\Omega)^2$,
$Pv_0=v_0$, and homogeneous Neumann boundary conditions. Then it holds
that
\begin{equation*}
\sup_{0\le t<S}\|v(t)\|_{\rH^1}^2
+\int_0^S\bigl(\|v\|_{\rH^2}^2+\|\dt v\|_2^2\bigr)\d t
\le C(T,h,M_T,\|v_0\|_{\rH^1}).
\end{equation*}
The same assertion holds with any positive initial time. In particular,
the bound is independent of $S$ and of higher derivatives of smooth
approximations.
\end{prop}

\begin{proof}
The proof is divided into six steps. We first establish the basic energy
bound, then estimate the baroclinic $\rL^6$ norm, the barotropic
$\rH^1$ norm, and the vertical derivative in turn. The fifth step
controls the horizontal and time derivatives, and the sixth passes from
smooth approximations to the given coefficients.

We work initially on $[0,t]$ with $t<S$, using smooth spatial
approximations and constants uniform in $t<S$. We write $C_M$ for
constants depending only on $h$ and $M_T$, and $C_T$ for constants
depending additionally on $T$ and $\|v_0\|_{\rH^1}$. To separate
the additional terms from the self-interaction, we define $f_Z(v)$ by
\begin{equation*}
f_Z(v)=-B(v,Z)-B(Z,v)-B(Z,Z).
\end{equation*}
The unprojected equation, with a pressure $\pi$ independent of $z$,
therefore reads
\begin{equation}
\label{eq-wind-energy-equation}
\dt v-\Delta v+B(v,v)+\nablaH\pi=f_Z(v).
\end{equation}
All spatial norms without an indicated domain are over $\Omega$.

\begin{step}[The basic energy estimate]
The Hardy inequality \eqref{eq-wind-hardy} and the identity
$w(b)/(-z)=-\mathcal H(\divH b)$ give
\begin{equation*}
\|w(Z)\|_\infty\le C_M
\ \text{ and } \ 
\|B(Z,Z)\|_3\le C_M.
\end{equation*}
For the second bound, both the horizontal product and the vertical
product $(w(Z)/(-z))R$ belong to $\rL_z^3\rL_\H^\infty$ by the
Hardy and H\"older inequalities. Since $K\le M_T$, the remainder
energy estimate from \autoref{lem-wind-relative-energy}, which includes
all mixed noise terms, yields
\begin{equation*}
\dt\|v\|_2^2+\|\nabla v\|_2^2\le C_M(1+\|v\|_2^2).
\end{equation*}
Gronwall's inequality therefore yields
\begin{equation}
\label{eq-wind-basic-energy-bound}
\sup_{0\le\tau\le t}\|v(\tau)\|_2^2
+\int_0^t\|\nabla v\|_2^2\d\tau\le C_T.
\end{equation}
\end{step}

\begin{step}[The baroclinic $\rL^6$ estimate]
\label{step-wind-baroclinic-six}
To estimate the fluctuation about the vertical mean, we define
\begin{equation*}
a=\overline v,\quad r=v-a,\quad s=Z-\overline Z,\quad
Y=\|r\|_6^6
\ \text{ and } \ 
D_6^2=\int_{\Omega}|r|^4|\nabla r|^2\d x.
\end{equation*}
Subtracting the averaged equation from
\eqref{eq-wind-energy-equation} gives
\begin{equation*}
\begin{aligned}
\dt r-\Delta r+a\cdot\nablaH r+r\cdot\nablaH a
+B(r,r)-\overline{B(r,r)}
={}&-a\cdot\nablaH s-s\cdot\nablaH a\\
&-B(r,Z)-B(Z,r)-B(Z,Z)\\
&+\overline{B(r,Z)+B(Z,r)+B(Z,Z)}.
\end{aligned}
\end{equation*}
Testing this pressure-free equation against $|r|^4r$ gives control
of $D_6^2$ through diffusion, while $a\cdot\nablaH r$ and
$B(r,r)$ cancel. For the remaining self-interactions, we apply the
horizontal Ladyzhenskaya inequality to
$(\int_{-h}^0|r|^6\d z)^{1/2}$ and the mean-tensor estimate
of~\cite[Section 3.2]{CT-07}, respectively, to obtain
\begin{equation*}
\begin{aligned}
|\int_{\Omega}(r\cdot\nablaH a)\cdot|r|^4r\d x|
&\le\eta D_6^2+C_\eta
\bigl(\|\nablaH a\|_{\rL^2(\T^2)}+\|\nablaH a\|_{\rL^2(\T^2)}^2\bigr)Y,\\
|\int_{\Omega}\overline{B(r,r)}\cdot|r|^4r\d x|
&\le\eta D_6^2+C_\eta
\bigl(\|\nablaH r\|_2^2+\|r\|_2^2\bigr)Y.
\end{aligned}
\end{equation*}
These estimates use only horizontal Sobolev inequalities and the zero
vertical flux, so the periodic setting introduces no additional term.

For the noise contributions, applying the three-dimensional
Gagliardo--Nirenberg inequalities to $|r|^3$, whose gradient has
$\rL^2$ norm at most $3D_6$, gives
\begin{equation}
\label{eq-wind-six-interpolations}
\|r\|_{36/5}^6\le C(Y+Y^{3/4}D_6^{1/2})
\ \text{ and } \ 
\|r\|_9^3\le C(Y^{1/2}+Y^{1/4}D_6^{1/2}).
\end{equation}
The horizontal part of $B(r,Z)$ therefore satisfies
\begin{equation*}
|\int_{\Omega}(r\cdot\nablaH Z)\cdot|r|^4r\d x|
\le\|\nablaH Z\|_6\|r\|_{36/5}^6
\le\eta D_6^2+C_{M,\eta}Y.
\end{equation*}
For its vertical part, the identity
$w(r)\dz Z=-\divH(\mathcal H r)R$ permits horizontal integration
by parts. H\"older's inequality and \eqref{eq-wind-hardy} give
\begin{equation*}
\begin{aligned}
|\int_{\Omega}w(r)\dz Z\cdot|r|^4r\d x|
&\le C\|\nablaH R\|_6\|\mathcal H r\|_{36/5}\|r\|_{36/5}^5
+C\|R\|_6D_6\|\mathcal H r\|_9\|r\|_9^2\\
&\le C\|\nablaH R\|_6\|r\|_{36/5}^6
+C\|R\|_6D_6\|r\|_9^3\\
&\le\eta D_6^2+C_{M,\eta}Y.
\end{aligned}
\end{equation*}
The last bound follows from \eqref{eq-wind-six-interpolations}, with
highest powers $4/3$ and $4$ of $\|\nablaH R\|_6$ and $\|R\|_6$,
respectively, and requires no unweighted norm of $\dz Z$.

The unaveraged transport $B(Z,r)$ cancels against $|r|^4r$, whereas
both mean corrections remain. Their conservative forms give
\begin{equation*}
\overline{B(r,Z)}_i=\sum_{j=1}^2\partial_{x_j}\overline{r_jZ_i}
\ \text{ and } \ 
\overline{B(Z,r)}_i=\sum_{j=1}^2\partial_{x_j}\overline{Z_jr_i}.
\end{equation*}
After horizontal integration by parts, each pairing contains an
averaged tensor and the average of $\nablaH(|r|^4r)$. Since vertical
averaging is a contraction on $\rL^\ell(\Omega)$ when its output
is extended constantly in $z$, H\"older's inequality with exponents
$18/5$ and $18/13$ gives
\begin{equation*}
|\int_{\Omega}\overline{B(r,Z)}\cdot|r|^4r\d x|
+|\int_{\Omega}\overline{B(Z,r)}\cdot|r|^4r\d x|
\le C\|Z\|_6D_6\|r\|_9^3
\le\eta D_6^2+C_{M,\eta}Y.
\end{equation*}
Indeed, the tensor is bounded by $\|Z\|_6\|r\|_9$ in
$\rL^{18/5}$, whereas
$\||r|^4\nablaH r\|_{18/13}\le D_6\|r\|_9^2$.

For the interactions with $a$, we define the column norm $\varphi$ by
$\varphi(x_\H)=(\int_{-h}^0|r(x_\H,z)|^6\d z)^{1/2}$.
The chain rule gives
\begin{equation*}
\|\varphi\|_{\rL^2(\T^2)}^2=Y
\ \text{ and } \ 
\|\nablaH\varphi\|_{\rL^2(\T^2)}\le3D_6.
\end{equation*}
H\"older's inequality first in $z$ and then in $x_\H$, followed by
the two-dimensional Gagliardo--Nirenberg inequality, consequently gives
\begin{equation*}
\begin{aligned}
|\int_{\Omega}(a\cdot\nablaH s+s\cdot\nablaH a)\cdot|r|^4r\d x|
&\le C_M\|a\|_{\rH^1(\T^2)}\|\varphi\|_{\rL^{10/3}(\T^2)}^{5/3}\\
&\le C_M\|a\|_{\rH^1(\T^2)}
(Y^{1/2}D_6^{2/3}+Y^{5/6})\\
&\le\eta D_6^2+C_{M,\eta}(1+\|a\|_{\rH^1(\T^2)}^2)(1+Y).
\end{aligned}
\end{equation*}
Finally, using the $\rL^3$ bounds for the pure noise and its average
and interpolating $|r|^3$ once more gives
\begin{equation*}
\begin{aligned}
|\int_{\Omega}\bigl(B(Z,Z)-\overline{B(Z,Z)}\bigr)\cdot|r|^4r\d x|
&\le C_M\|r\|_{15/2}^5\\
&\le C_M(Y^{7/12}D_6^{1/2}+Y^{5/6})
\le\eta D_6^2+C_{M,\eta}(1+Y).
\end{aligned}
\end{equation*}
Combining these estimates with small enough $\eta$ yields
\begin{equation*}
\dt Y+D_6^2
\le C_M(1+\|v\|_2^2+\|\nabla v\|_2^2)(1+Y).
\end{equation*}
Since $Y(0)\le C\|v_0\|_{\rH^1}^6$, the previously established
\eqref{eq-wind-basic-energy-bound} and Gronwall's inequality give
\begin{equation}
\label{eq-wind-six-bound}
\sup_{0\le\tau\le t}\|r(\tau)\|_6^6
+\int_0^tD_6^2\d\tau\le C_T.
\end{equation}
\end{step}

\begin{step}[The barotropic $\rH^1$ estimate]
The averaged equation has the same self-interaction terms as
in~\cite[Section 3.3.1]{CT-07}, with the additional force
$\overline{f_Z(v)}$. The conservative mean identities express this
force using only horizontal derivatives. With tensor components
$(b\otimes c)_{ij}=b_i c_j$ and divergence on the second index, they
give
\begin{equation*}
\overline{f_Z(v)}=-\divH\overline{Z\otimes v+v\otimes Z+Z\otimes Z}.
\end{equation*}
Differentiating the products horizontally and applying
Cauchy--Schwarz in $z$ gives
\begin{equation*}
\|\overline{f_Z(v)}\|_{\rL^2(\T^2)}
\le C_M(1+\|v\|_2+\|\nablaH v\|_2).
\end{equation*}
The known self-interaction force satisfies
\begin{equation*}
\|\divH\overline{r\otimes r}\|_{\rL^2(\T^2)}^2
\le C\int_{\Omega}|r|^2|\nablaH r|^2\d x
\le C(D_6^2+\|\nablaH r\|_2^2).
\end{equation*}
Since the self-transport cancels in the $\rH^1$ identity on the
two-dimensional torus, testing the averaged equation against
$-\Delta_\H a$ yields
\begin{equation*}
\dt\|\nablaH a\|_{\rL^2(\T^2)}^2+\|\Delta_\H a\|_{\rL^2(\T^2)}^2
\le C(D_6^2+\|\nablaH r\|_2^2)
+C_M(1+\|v\|_2^2+\|\nablaH v\|_2^2).
\end{equation*}
Using \eqref{eq-wind-basic-energy-bound} and
\eqref{eq-wind-six-bound}, we obtain
\begin{equation*}
\sup_{0\le\tau\le t}\|a(\tau)\|_{\rH^1(\T^2)}^2
+\int_0^t\|\Delta_\H a\|_{\rL^2(\T^2)}^2\d\tau\le C_T.
\end{equation*}
The embedding $\rH^1(\T^2)\hookrightarrow\rL^6(\T^2)$ and
\eqref{eq-wind-six-bound} now give
\begin{equation}
\label{eq-wind-full-six-bound}
\sup_{0\le\tau\le t}\|v(\tau)\|_6\le C_T.
\end{equation}
\end{step}

\begin{step}[The vertical derivative estimate]
To control the noise in the next energy identity, we write
$\zeta=\dz v$, $X=\|\nablaH v\|_2$ and
$M=\|\nablaH\zeta\|_2$. One-dimensional interpolation for
$\rL^2(\T^2)$-valued functions gives
\begin{equation}
\label{eq-wind-mixed-three}
\|g\|_{\rL_z^3\rL_\H^2}
\le C\|g\|_2^{5/6}(\|g\|_2+\|\dz g\|_2)^{1/6}.
\end{equation}
Together with the Hardy inequality, this yields
\begin{equation*}
\begin{aligned}
\|v\cdot\nablaH Z\|_2&\le C_M\|v\|_{\rL_z^3\rL_\H^2},\\
\|Z\cdot\nablaH v\|_2+\|w(v)\dz Z\|_2
&\le C_M\|\nablaH v\|_{\rL_z^3\rL_\H^2}
\ \text{ and } \ 
\|w(Z)\dz v\|_2\le C_M\|\zeta\|_2.
\end{aligned}
\end{equation*}
In particular, including the pure noise term, we obtain
\begin{equation}
\label{eq-wind-force-lower-order}
\|f_Z(v)\|_2
\le C_M\bigl(1+\|v\|_{\rH^1}+X^{5/6}(X+M)^{1/6}\bigr).
\end{equation}
We use this pointwise bound in the vertical energy identity to establish
the time integrability of $f_Z(v)$ below. Testing
\eqref{eq-wind-energy-equation} against $-\dz^2v$ eliminates the
pressure, and the standard self-interaction estimate
of~\cite[Section 3.3.2]{CT-07} gives
\begin{equation*}
|\int_{\Omega}B(v,v)\cdot\dz^2v\d x|
\le\eta\|\nabla\zeta\|_2^2
+C_\eta(\|v\|_6^2+\|v\|_6^4)\|\zeta\|_2^2.
\end{equation*}
For the noise, Young's inequality controls the only nonquadratic
expression from \eqref{eq-wind-force-lower-order} by giving
\begin{equation*}
C_MX^{5/6}\|\nabla\zeta\|_2^{7/6}
\le\eta\|\nabla\zeta\|_2^2+C_{M,\eta}X^2.
\end{equation*}
Thus the forcing pairing satisfies
\begin{equation*}
|\int_{\Omega}f_Z(v)\cdot\dz^2v\d x|
\le\eta\|\nabla\zeta\|_2^2
+C_{M,\eta}(1+\|v\|_2^2+X^2+\|\zeta\|_2^2).
\end{equation*}
After absorption, we obtain
\begin{equation*}
\dt\|\zeta\|_2^2+\|\nabla\zeta\|_2^2
\le C(C_M+\|v\|_6^2+\|v\|_6^4)\|\zeta\|_2^2
+C_M(1+\|v\|_2^2+X^2).
\end{equation*}
Gronwall's inequality, using \eqref{eq-wind-basic-energy-bound} and
\eqref{eq-wind-full-six-bound}, yields
\begin{equation}
\label{eq-wind-vertical-one-bound}
\sup_{0\le\tau\le t}\|\zeta(\tau)\|_2^2
+\int_0^t\|\nabla\zeta\|_2^2\d\tau\le C_T.
\end{equation}
Squaring and integrating \eqref{eq-wind-force-lower-order}, we control
the remaining mixed term by H\"older's inequality, which gives
\begin{equation*}
\int_0^tX^{5/3}M^{1/3}\d\tau
\le\left(\int_0^tX^2\d\tau\right)^{5/6}
\left(\int_0^tM^2\d\tau\right)^{1/6}\le C_T.
\end{equation*}
The basic and vertical energy estimates bound the first and second
factors, respectively, so without using a supremum bound for $X$, we
obtain
\begin{equation}
\label{eq-wind-force-time-two}
\int_0^t\|f_Z(v)\|_2^2\d\tau\le C_T.
\end{equation}
\end{step}

\begin{step}[The horizontal derivative and time derivative estimates]
For the final spatial estimate, we write
$J=\|\zeta\|_2$ and $D_\H=\|\nabla\nablaH v\|_2$.
Testing \eqref{eq-wind-energy-equation} against $-\Delta_\H v$
and using the self-interaction estimates of
\cite[Section 3.3.3]{CT-07} gives
\begin{equation*}
\dt X^2+D_\H^2
\le C\bigl(\|v\|_6^2+\|v\|_6^4+J(J+M)+J^2(J+M)^2\bigr)X^2
+C\|f_Z(v)\|_2^2.
\end{equation*}
Here the mixed norms in the vertical convection estimate are controlled
by the horizontal Ladyzhenskaya inequality, which gives
\begin{equation}
\label{eq-wind-horizontal-mixed-four}
\|w(v)\|_{\rL_z^\infty\rL_\H^4}
\le CX^{1/2}(X+D_\H)^{1/2}
\ \text{ and } \ 
\|\zeta\|_{\rL_z^2\rL_\H^4}
\le CJ^{1/2}(J+M)^{1/2}.
\end{equation}
The coefficient of $X^2$ is integrable by
\eqref{eq-wind-full-six-bound} and \eqref{eq-wind-vertical-one-bound},
and \eqref{eq-wind-force-time-two} controls the additive term.
Gronwall's inequality therefore gives
\begin{equation*}
\sup_{0\le\tau\le t}X(\tau)^2+\int_0^tD_\H^2\d\tau\le C_T.
\end{equation*}
Together with the vertical derivative estimate and the basic energy
bound, this controls every second spatial derivative and yields
\begin{equation*}
\sup_{0\le\tau\le t}\|v(\tau)\|_{\rH^1}^2
+\int_0^t\|v\|_{\rH^2}^2\d\tau\le C_T.
\end{equation*}
For the time derivative, interpolation and
\eqref{eq-wind-horizontal-mixed-four} give
\begin{equation*}
\|(v\cdot\nablaH)v\|_2^2\le C\|v\|_6^2X(X+D_\H)
\ \text{ and } \ 
\|w(v)\zeta\|_2^2\le CX(X+D_\H)J(J+M).
\end{equation*}
Both right-hand sides are integrable by the preceding bounds, with
Cauchy--Schwarz controlling the product $D_\H M$. Applying $P$ to
\eqref{eq-wind-energy-equation} and using
\eqref{eq-wind-force-time-two}, we obtain
\begin{equation*}
\int_0^t\|\dt v\|_2^2\d\tau
\le C\int_0^t\bigl(\|\Delta v\|_2^2+\|B(v,v)\|_2^2
+\|f_Z(v)\|_2^2\bigr)\d\tau\le C_T.
\end{equation*}
\end{step}

\begin{step}[Passage to the given coefficients]
For fixed $t<S$, approximate $v_0$ in $\rH^1$ by smooth
hydrostatically solenoidal Neumann data and approximate $Z$ as
in \autoref{prop-wind-regularization}. Applying the local strong
stability from \autoref{prop-wind-local-strong} successively along
the given solution on $[0,t]$ gives convergence of the approximate
solutions in $\rC\rH^1\cap\rL^2\rH^2$, with time derivatives
converging in $\rL^2\rL^2$. Compactness of the image of the given
solution in $\rH^1$ reduces this argument to finitely many local
intervals, so it uses only local existence and continuous dependence.
For the smooth spatial coefficients, spatial regularization or
truncation of $|r|^4r$ justifies the energy tests above without
differentiating the coefficients in time.

Convergence in both coefficient norms defining $M_T$ makes all
estimates uniform, and weak lower semicontinuity gives the quadratic
bounds. For the baroclinic dissipation, local strong convergence gives
almost-everywhere convergence of $r$ and $\nabla r$ along a
subsequence, identifying the weak $\rL^2$ limit of $|r|^2\nabla r$
and giving the corresponding lower semicontinuity for $D_6^2$.

The conservative mean identities also hold at the limiting regularity
without a trace of $\dz Z$. Since $v\in\rH^2$ at almost every
time, the product rules give
\begin{equation*}
\dz(w(v)Z)=-(\divH v)Z+(w(v)/(-z))R
\ \text{ and } \ 
\dz(w(Z)v)=-(\divH Z)v+w(Z)\dz v.
\end{equation*}
The Hardy inequality and \eqref{eq-wind-mixed-three} place both
products in $\rH_z^1\rL_\H^2$ and their quotients by $(-z)$ in
$\rL^2(\Omega)$. Their top traces therefore vanish, since a
nonzero trace of a continuous $\rL^2(\T^2)$-valued function is
incompatible with square integrability after division by $(-z)$.
At the bottom, $Z$ has an ordinary Sobolev trace because $(-z)$
is bounded away from zero, and $w(v)=w(Z)=0$. The same argument
applies to $w(Z)Z$ in $\rW_z^{1,3}\rL_\H^2$, whose derivative
and quotient by $(-z)$ belong to $\rL_z^3\rL_\H^2$. All vertical
flux terms therefore vanish, and the cancellations and mean identities
also follow by the approximation passage above.
The bounds are uniform for $t<S$, so letting $t$ increase to $S$
proves the assertion. \qedhere
\end{step}
\end{proof}
\section{Proof of the main theorem}
\label{sec-wind-proof}

We combine the local construction in the anisotropic maximal regularity
class with the global a priori estimate from a positive regular time.
The resulting continuation is classical strong at positive times,
without asserting global maximal regularity in the original pair
$(\rX_0,\rX_1)$.

\begin{proof}[Proof of \autoref{thm-global-wind}]
The proof is divided into five steps. We first construct the local
remainder and then continue it in the classical strong class. The
third and fourth steps establish independence of the restart time
and adaptedness, respectively. Finally, we verify the boundary
condition and prove pathwise uniqueness.

\begin{step}[The local construction]
Fix $T>0$ and $1/4<s<1/2$.
By \autoref{prop-wind-path-regularity} and
\autoref{prop-wind-regularization}, the stochastic convolution $Z$
has the required path regularity and spatial approximations almost surely.
Since $V_0\in P\rL_z^2\rH_\H^3\hookrightarrow\rX_\gamma$,
\autoref{prop-wind-local-hilbert} yields a unique maximal local
solution $v$ of \eqref{eq-remainder-anisotropic}, with lifetime
$t_+\in(0,T]$. For every $0<S<t_+$, it holds that
\begin{equation*}
v\in\rH^1(0,S,\rX_0)\cap\rL^2(0,S,\rX_1)
\cap\rC([0,S],\rX_\gamma).
\end{equation*}
The continuous and causal dependence of the local solution on $Z$,
together with localization of its adapted coefficient norms,
gives an adapted solution up to a strictly positive stopping time
$\tau<t_+$.
Setting $V=v+Z$ gives a local solution of
\eqref{eq-stochastic-pe} subject to
\eqref{eq-stochastic-boundary}.
By \autoref{lem-wind-relative-energy}, $v$ also belongs to the
energy class on every $[0,S]$ with $S<t_+$, which permits the
comparison with strong solutions below.
\end{step}
\begin{step}[Continuation in the classical strong class]
The embedding $\rX_1\hookrightarrow\rH^1(\Omega)^2$ gives
$v(a)\in P\rH^1$ for almost every $a\in(0,t_+)$ on each path.
For such an $a$, \autoref{prop-wind-local-strong} provides a local
strong solution $u^a$ with initial value $v(a)$ satisfying
\begin{equation*}
\dt u^a+Au^a=-F(u^a+Z,u^a+Z)
\ \text{ and } \ u^a(a)=v(a).
\end{equation*}
The stochastic convolution is thus unchanged at the restart time.
For every endpoint $S\le T$ within the strong existence interval,
\autoref{prop-wind-global-energy} gives
\begin{equation*}
\sup_{a\le t<S}\|u^a(t)\|_{\rH^1}^2
+\int_a^S\bigl(\|u^a\|_{\rH^2}^2+\|\dt u^a\|_2^2\bigr)\d t
\le C(T,h,M_T,\|v(a)\|_{\rH^1}).
\end{equation*}
Since this bound is independent of $S$, the continuation alternative
in \autoref{prop-wind-local-strong} extends $u^a$ through $T$, and
it holds that
\begin{equation*}
u^a\in\rH^1(a,T,\rL_{\bar\sigma}^2(\Omega))
\cap\rL^2(a,T,D(A))\cap\rC([a,T],P\rH^1(\Omega)^2).
\end{equation*}
For every $a<S<t_+$, the local remainder and $u^a$ have the same
left trace at $a$, so \autoref{lem-wind-relative-energy} gives
$u^a=v$ on $[a,S]$ and hence on $[a,t_+)$. We therefore define
the continuation $\widetilde v^a$ by
\begin{equation*}
\widetilde v^a(t)=
\begin{cases}
v(t),&0\le t\le a,\\
u^a(t),&a<t\le T.
\end{cases}
\end{equation*}
The common trace at $a$ and agreement on $[a,t_+)$ show that
$\widetilde v^a$ solves \eqref{eq-remainder} in $D(A)'$ throughout $(0,T)$
and attains $V_0$ at zero. Here consistency identifies $\sA v$ and $Au^a$
with $A_{-1}v$ and $A_{-1}u^a$, respectively, on their intervals of construction.

\end{step}

\begin{step}[Independence of the restart time]
Let $0<a<b<t_+$ be two regular times and let $u^a,u^b$ be their
strong continuations. Comparison with the original local solution
at time $b$ gives
\begin{equation*}
u^a(b)=v(b)=u^b(b).
\end{equation*}
Both continuations solve \eqref{eq-remainder-strong} with the same path $Z$
on $[b,T]$, so strong uniqueness from
\autoref{lem-wind-relative-energy} gives $u^a=u^b$ there. Before
$a$, both continued functions equal $v$, and between $a$ and $b$
the comparison $u^a=v$ applies, yielding
\begin{equation*}
\widetilde v^a=\widetilde v^b\quad\text{on }[0,T].
\end{equation*}
One regular restart time therefore determines the global remainder
independently of that choice. Since regular times can be chosen below
every fixed $\delta>0$, this remainder belongs to the classical
strong class on every $[\delta,T]$. Its agreement with the original
local solution also preserves the anisotropic maximal regularity on
$[0,\tau]$, while the initial energy estimate and the positive-time
strong bounds give the energy bound on $[0,T]$. We henceforth write
$v$ for the common continuation.
\end{step}

\begin{step}[Adaptedness of the continuation]
To make the restart-time selection measurable, we use the adapted
local solution on $[0,\tau]$. Every closed ball of $P\rH^1$ is
closed in $\rL_{\bar\sigma}^2$, because a sequence bounded in
$\rH^1$ and converging in $\rL^2$ has a weakly convergent
$\rH^1$ subsequence with the same limit. Thus $P\rH^1$ is
Borel in $\rL_{\bar\sigma}^2$, and the almost-everywhere
regularity together with Tonelli's theorem yields
\begin{equation*}
\int_0^T\PP\big(a<\tau,\ v(a)\notin P\rH^1\big)\d a=0.
\end{equation*}
We may therefore choose deterministic $a_n\downarrow0$ such that
$v(a_n)\in P\rH^1$ almost surely on
$E_n=\{a_n<\tau\}\in\sF_{a_n}$.
On $E_n$, the initial value $v(a_n)$ is measurable as an
$\rH^1$-valued variable.

Start the causal strong construction at $a_n$ with this value on
$E_n$ and zero initial value outside $E_n$. The global energy
bounds and successive local contractions give an adapted strong
continuation on $[a_n,T]$, with measurability of the limits following
from the continuous dependence in \autoref{prop-wind-local-strong}.
The continuations agree on overlapping events by the preceding step,
and $\bigcup_n E_n$ has probability one because $\tau>0$.
For fixed $t>0$, choosing the first $n$ with $a_n<t$ and $E_n$
is an $\sF_t$-measurable selection, so these agreeing continuations
define an adapted global process. Their initial and positive-time
continuities identify it with the pathwise continuation already
constructed.
\end{step}

\begin{step}[The boundary condition and uniqueness]
The local trace is attained in $\rX_\gamma$, and
$Z\in\rC([0,T],\rX_\gamma)$ with $Z(0)=0$, by
\autoref{prop-wind-path-regularity}. Thus $V=v+Z$ has initial value
$V_0$. For every smooth hydrostatically solenoidal Neumann test field
$\varphi$, adding \eqref{eq-lift-weak} to the integrated remainder
equation yields
\begin{equation*}
\begin{aligned}
(V(t),\varphi)_2+\int_0^t(V(r),A\varphi)_2\d r
+\int_0^t\langle F(V(r),V(r)),\varphi\rangle\d r
=(V_0,\varphi)_2
+\sum_{j=1}^N\beta_j(t)\int_{\T^2}g_j\cdot\varphi(\cdot,0)\d x_\H.
\end{aligned}
\end{equation*}
The nonlinear term is interpreted through the $\rX_0$ duality
pairing near zero and the representatives estimated above at positive
times, so this identity verifies system \eqref{eq-stochastic-pe}
subject to the boundary conditions \eqref{eq-stochastic-boundary}
in the stated sense. Testing with constant horizontal vectors also gives
\begin{equation*}
\int_{\Omega}V(t)\d x=\int_{\Omega}V_0\d x
+\sum_{j=1}^N\beta_j(t)\int_{\T^2}g_j\d x_\H,
\end{equation*}
which includes the momentum change caused by nonzero-mean wind profiles.

Suppose two solutions of system \eqref{eq-stochastic-pe} subject to
\eqref{eq-stochastic-boundary} have the same initial datum and Brownian
motions, and their remainders have the regularity asserted in
\autoref{thm-global-wind}. They have the same stochastic convolution
$Z$, and each remainder belongs to
$\rH^1(0,\tau_i,\rX_0)\cap\rL^2(0,\tau_i,\rX_1)$ on an
initial interval with $\tau_i>0$. Local uniqueness from
\autoref{prop-wind-local-hilbert} makes them agree on a common initial
interval. From any positive time in that interval, both remainders
are classical strong solutions with the same initial value, so
\autoref{lem-wind-relative-energy} propagates their equality through
$T$. This proves pathwise uniqueness from the original initial time
in the asserted regularity class, without assuming uniqueness for
arbitrary energy weak solutions.

The same uniqueness argument makes solutions on different finite
horizons agree on their common intervals, so the arbitrary choice
of $T$ gives a unique global process. The initial contraction
requires small time, while the global estimates require only finite
coefficient norms on finite intervals. Neither step imposes smallness
of the initial datum or the wind profiles. \qedhere
\end{step}
\end{proof}
 {\bf Statements and Declarations}

    \medskip

    {\bf Ai Usage Statement} Chatgpt's Astra has been used to verify \autoref{prop-wind-path-regularity}. The proof in its final form was developed by the authors, who take full responsibility for all claims and arguments made in the paper.

\end{document}